\documentclass[12pt,draft]{amsart}
\usepackage{amssymb,times}

\makeatletter
\def\@strippedMR{}
\def\@scanforMR#1#2#3\endscan{
  \ifx#1M\ifx#2R\def\@strippedMR{#3}
  \else\def\@strippedMR{#1#2#3}
  \fi\fi}
\renewcommand\MR[1]{\relax\ifhmode\unskip\spacefactor3000 \space\fi
  \@scanforMR#1\endscan
  MR\MRhref{\@strippedMR}{\@strippedMR}}
\makeatother

\parskip=\medskipamount

\newcommand{\mf}[1]{\mathbb{#1}}
\newcommand{\mc}[1]{\mathcal{#1}}

\DeclareMathOperator{\Var}{\mathrm{Var}}

\newcommand{\norm}[1]{\left\Vert#1\right\Vert}
\newcommand{\abs}[1]{\left\vert#1\right\vert}

\newcommand{\set}[1]{\left\{#1\right\}}

\newcommand{\eps}{\varepsilon}

\newcommand{\utimes}{\kern0.05em\buildrel{\times}\over{\rule{0em}{0.004em}}\kern-0.9em\cup \kern0.2em}
\newcommand{\sutimes}{\mathrel{\kern0em\buildrel{\mathsf{x}}\over{\rule{0em}{0.0em}} \kern-0.35em\cup\kern-0.0em}}

\newtheorem{theorem}{Theorem}[section]

\newtheorem{lemma}[theorem]{Lemma}

\newtheorem{proposition}[theorem]{Proposition}

\newtheorem{corollary}[theorem]{Corollary}

\theoremstyle{definition}

\newtheorem{remark}[theorem]{Remark}
\newtheorem{definition}[theorem]{Definition}

\allowdisplaybreaks[1]

\title[Limit theorems for monotonic convolution]{Limit theorems for monotonic convolution and the Chernoff product formula}
\author{Michael Anshelevich}
\thanks{The first author was supported in part by NSF grants DMS-0900935 and DMS-1160849.}
\address{Department of Mathematics, Texas A\&M University, College Station, TX 77843-3368}
\email{manshel@math.tamu.edu, jwilliams@math.tamu.edu}
\author{John D. Williams}
\subjclass[2010]{Primary 46L53; Secondary 30D05, 46L54, 47D06, 60B10}
\date{\today}

\begin{document}

\begin{abstract}
Bercovici and Pata showed that the correspondence between classically, freely, and Boolean infinitely divisible distributions holds on the level of limit theorems. We extend this correspondence also to distributions infinitely divisible with respect to the additive monotone convolution. Because of non-commutativity of this convolution, we use a new technique based on the Chernoff product formula. In fact, the correspondence between the Boolean and monotone limit theorems extends from probability measures to positive measures of total weight at most one. Finally, we study this correspondence for multiplicative monotone convolution, where the Bercovici-Pata bijection no longer holds.
\end{abstract}


\maketitle

\section{Introduction}

This article studies limit theorems for measures, but first we state a corollary which can be expressed purely in terms of analytic functions. Let
\[
\mc{A} = \set{F : \mf{C}^+ \rightarrow \mf{C}^+ \text{ analytic}, \lim_{y \uparrow \infty} F(iy)/(i y) = 1}.
\]
Note that $\mc{A}$ is closed under composition. We say that $F$ is \emph{infinitely divisible} if $F \in \mc{A}$ and for any $n \in \mf{N}$, there exists $g_n \in \mc{A}$ such that
\[
F = \underbrace{g_n \circ g_n \circ \ldots \circ g_n}_{n}.
\]

\begin{theorem} \cite[Proposition~3.8]{Belinschi-Thesis}
$F \in \mc{A}$ is infinitely divisible if and only if there exist $\set{F_t : t \geq 0} \subset \mc{A}$ which form a semigroup under composition,
\[
F_t \circ F_s = F_{t+s}, \quad F_1 = F,
\]
and is continuous in the topology of uniform convergence on compact sets. In this case we write $F^{\circ t} = F_t$; each $F^{\circ t}$ is uniquely defined.
\end{theorem}

Moreover, according to Proposition~\ref{Prop:Generator} below, there exists a function $\Phi$ with $\Phi(z) + z \in \mc{A}$ such that $\frac{\partial F_t}{\partial t} = \Phi(F_t)$. In terms of analytic functions, our main result is the following

\begin{corollary}
\label{Corollary:Functions}
Fix $\set{g_n : n \in \mf{N}}, F \in \mc{A}$, $F$ infinitely divisible, and a sequence of positive integers $k_1 < k_2 < \ldots$. Then
\[
\underbrace{g_n \circ g_n \circ \ldots \circ g_n}_{k_n} \rightarrow F
\]
uniformly on compact sets if and only if
\[
k_n \left( g_n(z) - z \right) \rightarrow \Phi,
\]
for $\Phi$ as above.
\end{corollary}

The proof of this corollary requires us to work with limit theorems for measures rather than analytic functions, which we now explain.

A fundamental result in free probability, due to Bercovici and Pata \cite{BerPatDomains}, is that limit theorems for sums of freely independent random variables are in a precise correspondence with limit theorems for independent random variables.
More specifically, denoting by $\boxplus$ the (additive) free convolution and by $\ast$ the usual convolution,
a $k_n$-fold convolution $\mu_n \boxplus \mu_n \boxplus \ldots \boxplus \mu_n$ converges to a limit \emph{if and only if} a $k_n$-fold convolution $\mu_n \ast \mu_n \ast \ldots \ast \mu_n$ converges to a limit.
The correspondence between the limit measures is known as the Bercovici-Pata bijection, which has a surprisingly concrete form based upon the L\'{e}vy-Hin\u{c}in representations of the various infinitely divisible measures.
In addition to this, the same authors proved that the same result also holds for the (additive) Boolean convolution $\uplus$.

According to \cite{SpeUniv,Ben-Ghorbal-Independence,Muraki-Quasi-universal,Muraki-Natural-products} in addition to the usual, free, and Boolean independence, the only other notion of non-commutative independence with a universal property is the monotonic independence of Muraki \cite{Muraki-Independence}. He defined monotone convolution $\rhd$ for compactly supported measures, and this operation was extended to general probability measures on $\mf{R}$ in \cite{Franz-Monotone-Boolean}. In this article, we are interested in limit theorems with respect to the monotone convolution.

Standard proofs of limit theorems for independent random variables use the method of characteristic functions, based on the observation that (the logarithm of) the Fourier transform is a linearizing transform for the convolution:
\[
\log \mc{F}_{\mu \ast \nu}(\theta) = \log \mc{F}_\mu(\theta) + \log \mc{F}_\nu(\theta).
\]
Free and Boolean convolutions also have linearizing transforms:
\[
\varphi_{\mu \boxplus \nu}(z) = \varphi_\mu(z) + \varphi_\nu(z); \qquad E_{\mu \uplus \nu} (z) = E_\mu(z) + E_\nu(z)
\]
(the notation will be defined in the following section), and their properties are used in the proof of the Bercovici and Pata results. However, the monotone convolution is not commutative, and as such cannot have a linearizing transform.
As a result, a very different approach is necessary to incorporate monotone convolution into this bijection.

In addition to the proof using characteristic functions, Feller (Section IX.7 of \cite{Feller-volume2}) gives an alternative proof of classical limit theorems using semigroups of operators and their generators. It is this approach, based on the Chernoff product formula, that works well for monotonic independence. Note that classical probability deals with commuting random variables, which demands a much simpler variant of these Chernoff style arguments than our non-commutative setting. The idea of using the full power of the Chernoff product formula in a probabilistic setting goes back to Goldstein \cite{Gol76,Goldstain-Correction} (see also \cite{Semigroup-Poisson}).

Besides the central limit theorem and the Poisson limit theorem \cite{Muraki-Independence,Hasebe-Saigo-Monotone-cumulants}, the only other limit theorems in the monotone case of which we are aware are Wang's results on the central limit theorem in the general (non-compactly supported) case \cite{Wang-Monotone-CLT} and on the strict domains of attraction of strictly stable distributions \cite{Wang-Monotone}.
As far as we know, our result is new even in the general Poisson case. In addition to these existing works, Uwe Franz and Takahiro Hasebe have informed us that they are currently developing related results.

Our main result is the addition of part (d) in the following theorem. Various measures $\nu^{\gamma, \sigma}$ are defined in the next section.

\begin{theorem}
\label{Theorem:Main}
Fix a finite positive Borel measure $\sigma$ on $\mathbb{R}$, a real number $\gamma$, a sequence of probability measures $\set{\mu_{n}}_{n \in \mf{N}}$, and a sequence of positive integers
$k_{1} < k_{2} < \cdots $ The following assertions are equivalent:
\begin{enumerate} \item
The sequence $\underbrace{\mu_{n} \ast \mu_{n} \ast \cdots \ast \mu_{n}}_{k_{n}}$ converges weakly to $\nu_{\ast}^{\gamma,\sigma}$;
 \item
The sequence $\underbrace{\mu_{n} \boxplus \mu_{n} \boxplus \cdots \boxplus \mu_{n}}_{k_{n}}$ converges weakly to $\nu_{\boxplus}^{\gamma,\sigma}$;
 \item
The sequence $\underbrace{\mu_{n} \uplus \mu_{n} \uplus \cdots \uplus \mu_{n}}_{k_{n}}$ converges weakly to $\nu_{\uplus}^{\gamma,\sigma}$;
 \item
The sequence $\underbrace{\mu_{n} \rhd \mu_{n} \rhd \cdots \rhd \mu_{n}}_{k_{n}}$ converges weakly to $\nu_{\rhd}^{\gamma,\sigma}$;
\item
The measures
\[
k_{n}\frac{x^{2}}{x^{2}+1}d\mu_{n}(x) \rightarrow \sigma
\]
weakly, and
\[
\lim_{n \uparrow \infty} k_{n} \int_{\mathbb{R}} \frac{x}{x^{2} + 1}d\mu_{n}(x) = \gamma.
\]
\end{enumerate}
\end{theorem}

The equivalence of items (a-c) and (e) is Theorem~6.3 in \cite{BerPatDomains}.

While the usual convolution is linear in each of its arguments, and so is defined for general positive or even signed measures, free convolution is naturally defined only for probability measures. We note that Boolean and monotone convolutions, defined in the complex-analytic setting rather than through the appropriate notion of independence, also can be treated as binary operations on positive measures. Our proof of the equivalence between parts (c) and (d) of Theorem~\ref{Theorem:Main} works, with very minor modifications, for the setting of positive measures of total weight at most one; we do not try to find the appropriate generalization of the statement in part (e) to this setting.

In contrast with the additive case, multiplicative convolutions arise by taking products of independent random variables as opposed to sums.
These forms of monotone convolution were defined and studied in \cite{Bercovici-Multiplicative-monotonic,Franz-Multiplicative-monotone}.
A version of the Boolean-free Bercovici-Pata bijection for the multiplicative case was proven by Wang in \cite{Wang-Boolean}.
In Section 4, we investigate the analog of Theorem \ref{Theorem:Main} for multiplicative convolution. We show that the direct analog of the theorem does not hold in general, but holds under additional conditions. Note that we only consider multiplicative convolution for measures on the unit circle, and not on the positive real line as for example in \cite{Bercovici-Pata-Products}.

The paper is organized as follows.
Section 1 consists of this introduction.  Section 2 consists of preliminaries for additive non-commutative probability as well as the semigroup theory applicable to our proof of Theorem~\ref{Theorem:Main}.
In Section 3, we prove our main result---the equivalence between the Boolean and monotone limit theorems---first for measures with finite variance (by functional-analytic methods) and then for general positive measures of total weight at most one (by complex-analytic methods); Theorem~\ref{Theorem:Main} follows as a corollary. Section 4 consists of the preliminaries for multiplicative convolutions as well as multiplicative analogues of our main results.

\textit{Acknowledgements:}  The authors would like to thank J.C. Wang for reading this paper thoroughly and providing excellent advice during the revision process, and Serban Belinschi for some remarks.

\section{Preliminaries}

\subsection{Transforms and distributions}

In what follows, we shall denote by $$\mathbb{C}^{+} = \{z \in \mathbb{C} : \Im{(z)} > 0 \}, \ \mathbb{C}^{-} = \{z \in \mathbb{C} : \Im{(z)} < 0 \}, \ \mathbb{C}^{+}_{r} = \{z \in \mathbb{C} : \Im{(z)} > r \}$$ for $r \in \mathbb{R}^{+}$.
For $\alpha, \beta > 0$ we shall refer to the truncated cone $\Gamma_{\alpha ,\beta} = \{z \in \mathbb{C}^{+}_{\beta} : \Im{(z)} > \alpha |\Re{(z)}| \}$ as the \textit{Stolz angle} associated to these real numbers.

Let $\mu$ and $\nu$ denote finite (positive) non-zero Borel measures on the real line.  The \textit{Cauchy transform} associated to such a measure is the function
$$G_{\mu}(z) := \int_{\mathbb{R}} \frac{1}{z-x}d\mu(x) : \mathbb{C}^{+} \rightarrow \mathbb{C}^{-} \cup \mathbb{R} $$
We define the \textit{F-transform} associated to this measure
by letting $F_{\mu}(z) := 1/G_{\mu}(z)$.  Finally, there exist $\alpha , \beta > 0$ such that $F_{\mu}$ is injective when restricted to $\Gamma_{\alpha ,\beta}$
and we define the \textit{Voiculescu transform} by setting $$\varphi_{\mu}(z) := F_{\mu}^{-1}(z) - z : \Gamma_{\alpha ' , \beta '} \rightarrow \mathbb{C}^{-} \cup \mathbb{R}$$
where this function takes on real values if and only if the associated measure is a Dirac mass.  The Voiculescu transform may be viewed as an analogue of the logarithm of the
Fourier transform for free probability insofar as $\varphi_{\mu \boxplus \nu} = \varphi_{\mu} + \varphi_{\nu}$.

We define the \textit{E-transform} of a finite, non-zero Borel measure $\mu$ as $$E_{\mu}(z) := \frac{1}{\mu(\mf{R})} z - F_{\mu}(z). $$
For finite non-zero Borel measures $\mu$ and $\nu$, we may define their Boolean convolution $\mu \uplus \nu$ by requiring that
\[
E_{\mu \uplus \nu} = E_{\mu} + E_{\nu}, \qquad (\mu \uplus \nu)(\mf{R}) = \mu(\mf{R}) \nu(\mf{R});
\]
the existence of a positive measure $\mu \uplus \nu$ follows from Nevanlinna theory.
Observe that, in contrast to the free case,
the E-transform is well defined on all of $\mathbb{C}^{+}$.  This fact may be used to prove that all Borel probability measures are infinitely divisible (in the sense defined below) with respect to Boolean convolution.

Let $\{ \mu_{i} \}_{i \in I}$ denote a family of finite positive measures.  We say that this family it \textit{tight} if for every $\epsilon > 0$ there exists an $N \in \mathbb{N}$ so that
$\mu_{i}([-N,N]^c) < \epsilon$ for all $i \in I$.  It is a basic result in measure theory that a family of measures is uniformly bounded and tight if and only if this family is sequentially precompact in the weak topology. All the measures below will have total weight at most one, so we will omit the uniform boundedness condition.

$\mu_n \rightarrow \mu$ \emph{vaguely} if $\int f \,d\mu_n \rightarrow \int f \,d\mu$ for every $f \in C_0(\mf{R})$. This is equivalent to the uniform convergence of the Cauchy transforms, or of the $F$ transforms, on compact subsets of $\mf{C}^+$, and implies that $\mu(\mf{R}) \leq \liminf_{n \rightarrow \infty} \mu_n(\mf{R})$. $\mu_n \rightarrow \mu$ \emph{weakly} if $\int f \,d\mu_n \rightarrow \int f \,d\mu$ for every bounded continuous $f$, and is equivalent to vague convergence plus $\mu_n(\mf{R}) \rightarrow \mu(\mf{R})$, or to this last condition plus the uniform convergence on compact sets of the $E$-transforms. Thus when $\mu_n$ and $\mu$ are all probability measures, the two modes of convergence are equivalent.

We say that a measure $\mu$ is \textit{infinitely divisible} with respect to the convolution operation $\boxplus$ if, for all $n \in \mathbb{N}$
there exists a Borel probability measure $\mu_{n}$ such that $\mu = \mu_{n} \boxplus \mu_{n} \boxplus \cdots \boxplus \mu_{n}$ where the convolution is $n$-fold. Analogous definitions serve for all of the convolution operations discussed in this paper. It is known that a $\boxplus$-infinitely divisible measure $\mu$ can be included as $\mu_1$ in a $\boxplus$-convolution semigroup
\[
\set{\mu_t : t \geq 0}, \quad \mu_t \boxplus \mu_s = \mu_{t+s},
\]
and this property also holds for the other convolution operations discussed in this paper.

In what follows, we shall make liberal use of the following basic function theoretic facts.  We refer to \cite{BV93} for an excellent overview of the relevant machinery.
\begin{lemma}
\label{Lemma:Basic-facts}
$ $
\begin{enumerate}
 \item $\Im{(F_{\mu}(z))} \geq \frac{1}{\mu(\mf{R})} \Im{(z)}$ with equality at any point $z$ if and only if $\mu$ is a Dirac mass.

\item An analytic function $F : \mathbb{C}^{+} \rightarrow \mathbb{C}^{+} \cup \mathbb{R}$ is the $F$-transform of a Borel measure $\mu$, with $\lim_{y \uparrow \infty} F_{\mu}(iy)/(iy) = \frac{1}{\mu(\mf{R})}$.

\item $\abs{F_{\mu_{i}}(z) - \frac{1}{\mu_i(\mf{R})} z} = o(|z|)$ uniformly for $z \in \Gamma_{\alpha ,\beta}$ and $\{ \mu_{i} \}_{i \in I}$ a uniformly bounded, tight family of measures.

\item There exists a finite measure $\sigma$ and a real number $\gamma$ such that $$F_{\mu}(z) = - \gamma + \frac{1}{\mu(\mf{R})} z + \int_{\mathbb{R}} \frac{1+xz}{x-z}\,d\sigma(x)$$
\end{enumerate}
\end{lemma}
The last of these refers to the \textit{Nevanlinna representation} of certain complex analytic functions.
The following lemma is a slight reformulation of the results from \cite{Maa92}.

\begin{lemma}
\label{Lemma:Maassen}
$ $
\begin{enumerate}
\item
Let $\rho$ be a finite measure on $\mf{R}$. The Cauchy transform $G_\rho(z)$ is a bounded function on $\mf{C}_1^+$.
\item
If $\rho$ has a finite first moment, then $z G_\rho(z)$ and $z G_\rho'(z)$ are bounded functions on $\mf{C}_1^+$.
\item
Let $\mu$ be a finite measure with finite variance, which for a non-probability measure means $\mu(x^2) \mu(\mf{R}) - \mu(x)^2 < \infty$. Then $\frac{1}{\mu(\mf{R})} z - F_\mu(z)$ is a bounded function on $\mf{C}_1^+$. Moreover, if $\sup \set{\Var[\mu_n] : n \in \mf{N}} < \infty$, the bound is uniform in $n$.
\end{enumerate}
\end{lemma}

\begin{proof}
For finite $\rho$,
\[
\abs{\int_{\mf{R}} \frac{1}{z - x} \,d\rho(x)} \leq \frac{1}{\Im (z)} \rho(\mf{R}).
\]
For $\rho$ with a finite first moment and $z \in \mf{C}_1^+$,
\[
\abs{\int_{\mf{R}} \frac{z}{(z - x)^2} \,d\rho(x)}
\leq \abs{\int_{\mf{R}} \frac{z}{z - x} \,d\rho(x)}
\leq \rho(\mf{R}) + \frac{1}{\Im (z)} \int_{\mf{R}} \abs{x} \,d\rho(x).
\]
Finally, for $\mu$ a probability measure with finite variance, by Proposition~2.2 of \cite{Maa92},
\[
\frac{1}{\mu(\mf{R})} z - F_\mu(z) = \frac{\mu(x)}{\mu(\mf{R})^2} + G_\sigma(z),
\]
where $\sigma$ is a finite measure. In fact, $\sigma(\mf{R}) = \frac{\mu(x^2) \mu(\mf{R}) - \mu(x)^2}{\mu(\mf{R})^2}$, which implies the last result.
\end{proof}

\begin{remark}
\label{Remark:Levy}
The classical L\'{e}vy-Hin\u{c}in formula provided an equivalent definition of infinite divisibility based on the class of characteristic functions
associated to these measures.  Related formulae were developed for free and Boolean independence.  In defining these formulae, let $\gamma \in \mathbb{R}$ and $\sigma$ denote a finite Borel measure on $\mf{R}$ and define
the measures $\nu_{\ast}^{\gamma,\sigma}$ (resp. $\nu_{\boxplus}^{\gamma,\sigma}$ ; $\nu_{\uplus}^{m, \gamma,\sigma}$ ; $\nu_{\uplus}^{\gamma,\sigma} = \nu_{\uplus}^{1, \gamma,\sigma}$) in terms of the relevant transforms
by letting

\[
(\mc{F}_{\nu_\ast^{\gamma, \sigma}})(t) = \exp \left( i \gamma t + \int_{\mf{R}} (e^{i t x} - 1 - i t x) \frac{x^2 + 1}{x^2} \,d\sigma(x) \right), \quad t \in \mf{R},
\]
\[
\varphi_{\nu_{\boxplus}^{\gamma, \sigma}}(z) = \gamma + \int_{\mf{R}} \frac{1 + x z}{z - x} \,d\sigma(x), \quad z \in \mf{C}^+,
\]
\[
-E_{\nu_{\uplus}^{m, \gamma, \sigma}}(z) = F_{\nu_{\uplus}^{m, \gamma, \sigma}}(z) - \frac{1}{m} z = - \gamma + \int_{\mf{R}} \frac{1 + x z}{x-z} \,d\sigma(x), \quad z \in \mf{C}^+.
\]
A Borel probability measure $\mu$ is infinitely divisible with respect to classical (resp. free ; Boolean) convolution \textit{if and only if}
there exists a $\gamma$ and $\sigma$ as above so that $\mu = \nu_{\ast}^{\gamma,\sigma}$ (resp. $\mu = \nu_{\boxplus}^{\gamma,\sigma}$; $\mu = \nu_{\uplus}^{\gamma,\sigma}$)
 We shall define the class of
monotone infinitely divisible measures, which we shall denote by
$\nu_{\rhd}^{\gamma, \sigma}$, below.
\end{remark}

\begin{remark}
The reader should note that the classes of infinitely divisible probability measures are all indexed by a real number $\gamma$ and a finite measure $\sigma$.
That this bijection is more than formal is the main content of Theorem \ref{Theorem:Main}.
\end{remark}

\subsection{Monotonic independence and monotone convolution}

The notion of monotonic independence is originally due to Muraki (see \cite{Muraki-Independence} and references therein). In \cite{Muraki-Convolution}, he defined the corresponding convolution operation $\rhd$ on compactly supported probability measures. This definition was extended to general probability measures by Franz \cite{Franz-Monotone-Boolean}. Up to a change in notation, their definition amounts to requiring that
\[
F_{\mu \rhd \nu}(z) = F_\mu (F_\nu(z)).
\]
It follows immediately from Nevanlinna theory that the same definition works for general non-zero positive measures $\mu$ and $\nu$, and $(\mu \rhd \nu)(\mf{R}) = \mu(\mf{R}) \nu(\mf{R})$.

\subsection{Monotone convolution semigroups}

See Theorem 4.5 of \cite{BiaProcesses}, based on Section 5.2 of \cite{Montgomery-Zippin}, for the proof of the following result, and \cite{Berkson-Porta,Siskakis-review,Hasebe-Monotone-semigroups} for related results.

\begin{proposition}
\label{Prop:Generator}
Let  $\set{\nu_t : t \geq 0}$ form a monotone convolution semigroup (so that $\nu_t \rhd \nu_s = \nu_{t+s}$), which in particular is strongly continuous. Then, denoting $F_t = F_{\nu_t}$, the family $\set{F_t : t \geq 0}$ form a semigroup of analytic transformations of $\mf{C}^+$, which extends to a local group of analytic transformations of some $\Gamma_{\alpha, \beta}$ for some  $\eps$ and $-\eps \leq t \leq \eps$. If $m = \nu_1(\mf{R}) \leq 1$, then $\Im(F_t(z)) \geq \frac{1}{m^t} \Im(z) \geq \Im(z)$. Therefore there exists an analytic function $\Phi : \mf{C}^+ \rightarrow \mf{C}^+ \cup \mf{R}$ such that
\begin{equation}
\label{Eq:Evolution}
\frac{\partial F_t}{\partial t} = \Phi(F_t).
\end{equation}
By Nevanlinna theory $\Phi$ has a representation
\[
\Phi(z) = - \gamma - \log(m) z + \int_{\mathbb{R}} \frac{1+xz}{x-z}d\sigma(x)
\]
for a finite measure $\sigma$ and real numbers $0 < m \leq 1$ and $\gamma$. Conversely, given such $\Phi: \mf{C}^+ \rightarrow \mf{C}^+$, the equation \eqref{Eq:Evolution} has a unique solution with initial condition $F_{0}(z) = z$  corresponding to a strongly continuous monotone convolution semigroup.
\end{proposition}

\begin{definition}
\label{Defn:Monotone-nu}
For a finite measure $\sigma$ and real numbers $0 < m \leq 1$ and $\gamma$, denote
\[
\Phi^{m, \gamma, \sigma}(z) := - \gamma - \log (m) z + \int_{\mf{R}} \frac{1 + x z}{x-z} \,d\sigma(x), \quad z \in \mf{C}^+.
\]
It should be noted that $\Phi^{m, \gamma,\sigma} = -E_{\nu^{\gamma,\sigma}_{\uplus}} - \log(m) z$.
Let $\set{\nu_t : t \geq 0}$ be the monotone convolution semigroup it generates in the sense of the preceding proposition. Denote
\[
\nu_{\rhd}^{m, \gamma, \sigma} = \nu_1.
\]
\end{definition}

\begin{lemma}
\label{Lemma:Distance}
For $i = 1, 2$, let $\set{F_i(z,t) : z \in \mf{C}^+, t \geq 0}$ be two semigroups of analytic transformations of $\mf{C}^+$, such that $\frac{\partial F_i(z,t)}{\partial t} = \Phi_i(F_{i}(z,t))$, $F_i(z,0) = z$. Given $\eps > 0$, suppose that for some compact $K$ and $C$, we have $F_i(K,t) \subset C$ for $t \in [0,1]$, and for $z \in C$, $\abs{\Phi_1(z) - \Phi_2(z)} < \eps$. Then for any $z \in K$, $\abs{F_1(z,1) - F_2(z,1)} \leq c \eps$, where $c$ is a constant depending on $\Phi_2$ but not on $\Phi_1$.
\end{lemma}

\begin{proof}
Note first that
\[
\frac{\partial^2 F_i(z,t)}{\partial t^2} = \Phi_i'(F_i(z,t)) \Phi_i(F_i(z,t)),
\]
so in particular this second derivative exists. Then denoting
\[
M_2 = \max_{i=1,2} \sup_{z \in K} \sup_{t \in [0,1]} \abs{\frac{\partial^2 F_i(z,t)}{\partial t^2}},
\]
by Taylor's formula
\[
\abs{F_i(z,t) - F_i(z,t_0) - (t - t_0) \Phi_i(F_i(z, t_0))} \leq \frac{M_2}{2} (t - t_0)^2.
\]
Thus
\[
\abs{F_i(z, t_0 + 1/N) - F_i(z,t_0) - \frac{1}{N} \Phi_i(F_i(z, t_0))} < \frac{M_2}{2 N^2}.
\]
So denoting $M_1 = \sup_{z \in C} \abs{\Phi_2'(z)}$, it follows that
\[
\begin{split}
& \abs{F_1(z, t_0 + 1/N) - F_2(t_0 + 1/N)} \\
&\qquad \leq \frac{1}{N} \left( \sup_{z \in C} \abs{\Phi_1(z) - \Phi_2(z)} + (N + \sup_{z \in K} \abs{\Phi_2'(z)}) \abs{F_1(z, t_0) - F_2(z, t_0)} \right) + \frac{M_2}{N^2} \\
&\qquad \leq \frac{\eps}{N} + \frac{M_2}{N^2} + \left( 1 + \frac{M_1}{N} \right) \abs{F_1(z, t_0) - F_2(z, t_0)}.
\end{split}
\]
Therefore
\[
\begin{split}
\abs{F_1(z, 1) - F_2(1)}
& \leq \left( \frac{\eps}{N} + \frac{M_2}{N^2} \right) \sum_{k=0}^{N-1} \left( 1 + \frac{M_1}{N} \right)^k \\
& \approx (e^{M_1} - 1) \left( \frac{\eps}{M_1} + \frac{M_2}{M_1 N} \right)
\approx \frac{e^{M_1} - 1}{M_1} \eps
\end{split}
\]
for large $N$.
\end{proof}

\begin{remark}
We note here that weak convergence of a uniformly bounded family of measures is equivalent to the convergence of their $F$-transforms uniformly on compact sets (see \cite{BV93}).
We will use this fact without reference throughout the paper.  With this in mind, consider the  monotone infinitely divisible measures
$\{\mu_{n} \}_{n \in \mathbb{N}}$ and $\mu$ with associated semigroup generators $\{ \Phi_{n} \}_{n \in \mathbb{N}}$ and $\Phi$, respectively.  If we assume that $\Phi_{n} \rightarrow \Phi$ uniformly on compact sets
then, according to the previous lemma, $\mu_{n} \rightarrow \mu$ vaguely.  This fact will play a key role in the proof of our main theorem.
\end{remark}

\subsection{Chernoff product formula}

We will use the following version of the Chernoff Product Formula.

\begin{proposition}
\label{Prop:Chernoff}
Let $(k_n)$ be an increasing sequence of positive integers, and $\set{V_n}_{n \in \mf{N}}$ a family of contractions on a Banach space $\mc{X}$. Suppose $B$ is an unbounded operator which generates a strongly continuous semigroup of contractions $\set{T(t) : t \geq 0}$, $\mc{D}$ is a core for $B$, and for each $x \in \mc{D}$,
\[
\lim_{n \uparrow \infty} k_n (V_n - I) x \rightarrow B x.
\]
Then for each $x \in \mc{X}$,
\[
\lim_{n \uparrow \infty} V_n^{k_n} x = T(1) x.
\]
\end{proposition}

The proof is very similar to the continuous version, Theorem~5.2 of \cite{Engel-Nagel-book}. It is provided for completeness.

\begin{proof}
Denote $B_n = k_n( V_n - I)$. Then $B_n x \rightarrow B x$ for all $x \in \mc{D}$. Moreover,
\[
\norm{e^{t B_n}} \leq e^{-t k_n} \norm{e^{t k_n V_n}}
\leq e^{-t k_n} \sum_{j=0}^\infty \frac{t^j k_n^j \norm{V_n^j}}{j!} \leq 1,
\]
so the semigroups $\set{e^{t B_n} : t \geq 0}$ are contractive. Therefore by the first Trotter-Kato Approximation Theorem (Theorem~4.8 in \cite{Engel-Nagel-book}), $e^{B_n t} \rightarrow T(t)$ strongly and, in particular,
\[
\norm{e^{B_n} x - T(1) x} \rightarrow 0
\]
for each $x \in \mc{X}$, as $n \uparrow \infty$. On the other hand, by Lemma~5.1 from \cite{Engel-Nagel-book},
\[
\norm{e^{B_n} x - V_n^{k_n} x}
= \norm{e^{k_n(V_n - I)} x - V_n^{k_n} x}
\leq \sqrt{k_n} \norm{V_n x - x}
= \frac{1}{\sqrt{k_n}} \norm{B_n x} \rightarrow 0.
\]
The result follows.
\end{proof}

\section{Limit theorems for monotone convolution on $\mf{R}$}

In this section we prove our main theorem.  The Chernoff product formula is at the heart of the proof of the forward direction ($a,b,c,e \Rightarrow d$ in the parlance of Theorem \ref{Theorem:Main}).
The reader should note that, to satisfy the requisite hypotheses, we must make additional moment assumptions on the relevant random variables (these assumptions will later be discarded).

Denote by $\mc{M}$ the \emph{sub-probability} measures, that is, positive Borel measures with total weight at most one, by $\mc{M}_2 \subset \mc{M}$ the subset of measures with finite variance, and by $\mc{M}_1$ general finite positive Borel measures on $\mf{R}$ with finite first moment.

\begin{proposition}
\label{Prop:Variance}
Let $\set{\mu_n}_{n \in \mf{N}} \subset \mc{M}_2$ be a family of sub-probability measures on $\mf{R}$. Suppose
\[
\mu_n^{\uplus k_n} \rightarrow \nu_{\uplus}^{m, \gamma, \sigma}
\]
weakly, where $\{ \mu_{n}^{\uplus k_{n}} \}_{n=1}^{\infty}$ and $\nu_{\uplus}^{m, \gamma, \sigma}$ have uniformly bounded,  finite variance.
Then
\[
\mu_n^{\rhd k_n} \rightarrow \nu_{\rhd}^{m, \gamma, \sigma}
\]
weakly.
\end{proposition}

\begin{proof}
Let
\[
\mc{D} = \set{G_{\rho}: \rho \in \mc{M}_1}.
\]
By Nevanlinna theory, $\mc{D}$ is invariant under composition operators by functions $F_\mu$ for $\mu \in \mc{M}_2$. Let $\mc{A}$ be the completion of $\mc{D}$ with respect to the uniform norm on $\mf{C}_1^+$, which we denote by $\norm{\cdot}_\infty$.
Then each right composition operator by $F_\mu$ is a contraction on $\mc{A}$ (since $F_{\mu}(\mathbb{C}^{+}_{1}) \subset \mathbb{C}^{+}_{1}$).

We will utilize Chernoff's theorem (Proposition \ref{Prop:Chernoff}).
Towards this end, we define operators on $\mathcal{A}$ by letting  $$V_{n} \cdot h := h \circ F_{\mu_{n}}$$ for $h \in \mathcal{A}$.  We further define
a (possibly unbounded) operator $$B \cdot h := \Phi^{m, \gamma,\sigma}h'$$ using the notation from Definition~\ref{Defn:Monotone-nu}.  Note that the notation for the operators match those within the statement of Proposition \ref{Prop:Chernoff}.
To invoke this result, we must show that $\mathcal{D}$ is a core for the operator $B$. This will follow once we show that $\mathcal{D}$ is in the domain of this operator.

Towards this end, let $h = G_\rho \in \mc{D}$.  We will show that $\Phi^{m, \gamma,\sigma}h'$ is a limit of elements in $\mathcal{A}$, proving that $\Phi^{m, \gamma,\sigma}h' \in \mathcal{A}$.  Indeed,
\[
\begin{split}
& \norm{k_n(h \circ F_{\mu_n} - h) - \Phi^{m, \gamma, \sigma} h'}_\infty \\
&\qquad = \sup_{z \in \mf{C}_1^+} \abs{\int_{\mf{R}} \left(k_n \left( \frac{1}{F_{\mu_n}(z) - x} - \frac{1}{z - x} \right) + \frac{\Phi^{m, \gamma, \sigma}(z)}{(z - x)^2} \right) \,d\rho(x)} \\
&\qquad = \sup_{z \in \mf{C}_1^+} \abs{\int_{\mf{R}} \left(\frac{k_n(z - F_{\mu_n}(z)) + \Phi^{m, \gamma, \sigma}(z)}{(F_{\mu_n}(z) - x) (z - x)} + \frac{(F_{\mu_n}(z) - z) \Phi^{m, \gamma, \sigma}(z)}{(F_{\mu_n}(z) - x)(z - x)^2} \right) \,d\rho(x)} \\
&\qquad \leq \sup_{z \in \mf{C}_1^+} \abs{\left(k_n(z - F_{\mu_n}(z)) + \Phi^{m, \gamma, \sigma}(z)\right) G_\rho(z)}
+ \sup_{z \in \mf{C}_1^+} \abs{(F_{\mu_n}(z) - z) \Phi^{m, \gamma, \sigma}(z) G_\rho'(z)} \\
&\qquad \leq \sup_{z \in \mf{C}_1^+} \abs{\left(F_{\mu_{n}^{\uplus_{k_{n}}}}(z) - F_{\nu_{\uplus}^{m, \gamma, \sigma}}(z) \right) G_\rho(z)} + \sup_{z \in \mf{C}_1^+} \abs{\left(k_n \left( 1 - \frac{1}{\mu_n(\mf{R})} \right) - \log (m) \right) z G_\rho(z)} \\
&\qquad\quad + \sup_{z \in \mf{C}_1^+} \abs{(F_{\mu_n}(z) - z) \Phi^{m, \gamma, \sigma}(z) G_\rho'(z)}.
\end{split}
\]
The hypothesis implies that $\mu_n(\mf{R})^{k_n} \rightarrow \nu_{\uplus}^{m, \gamma, \sigma}(\mf{R}) = m$, and in particular $\mu_n(\mf{R}) \rightarrow 1$. It also implies that $\mu_n^{\uplus k_n} \rightarrow \nu_{\uplus}^{m, \gamma, \sigma}$ and $\mu_n \rightarrow \delta_0$ weakly, and so  on any compact set $$\left(E_{\mu_{n}^{\uplus_{k_{n}}}}(z) - E_{\nu_{\uplus}^{m, \gamma, \sigma}}(z) \right), \ \ \ (F_{\mu_n}(z) - z)$$ converge to zero uniformly, and in addition
\[
k_n \left( 1 - \frac{1}{\mu_n(\mf{R})} \right) \approx - k_n \log \left( \frac{1}{\mu_n(\mf{R})} \right) = \log \mu_n(\mf{R})^{k_n} \rightarrow \log m.
\]
On the other hand, the same convergence results imply that the variances of $\set{\mu_{n}^{\uplus_{k_{n}}}}_{n=1}^\infty$ and $\set{\mu_n}_{n=1}^\infty$ are all uniformly bounded, and so by Lemma~\ref{Lemma:Maassen}, the functions $$\left(E_{\mu_{n}^{\uplus_{k_{n}}}}(z) - E_{\nu_{\uplus}^{m, \gamma, \sigma}}(z) \right), \ \ \ (F_{\mu_n}(z) - z) \Phi^{\gamma, \sigma}(z)$$ are bounded uniformly in $n$. Finally, by the same lemma $z G_\rho(z)$ is bounded and $G_\rho(z) = o(1)$, $G_\rho'(z) = o(1)$. Combining these results, we conclude that
$$\norm{k_n(h \circ F_{\mu_n} - h) - \Phi^{m, \gamma, \sigma} h'}_\infty \rightarrow 0.$$

We have just shown that
\[
k_n (V_n - I) h \rightarrow B h
\]
for $h \in \mc{D}$, so that $\mc{D}$ is indeed in the domain of $B$.
According to \cite{Berkson-Porta,Siskakis-review}, the operator $B$ is precisely the generator of the semigroup of composition operators corresponding to the semigroup of functions generated by $\Phi^{m, \gamma,\sigma}$. This semigroup is strongly continuous.
Since these composition operators preserve $\mc{D}$, we may conclude that $\mc{D}$ is a core for $B$. It now follows from Proposition~\ref{Prop:Chernoff} that $V_n^{k_n} \rightarrow T(1)$ strongly, where \[
T(1) h = h \circ F_{\nu_1} = h \circ F_{\nu_{\rhd}^{m, \gamma, \sigma}}.
\]
In particular, for $h(z) = \frac{1}{z}$
\[
\norm{G_{\mu_n^{\rhd k_n}} - G_{\nu_{\rhd}^{m, \gamma, \sigma}}}_\infty \rightarrow 0,
\]
which implies that
\[
\mu_n^{\rhd k_n} \rightarrow \nu_{\rhd}^{m, \gamma, \sigma}
\]
weakly.
\end{proof}

The converse to the Chernoff product formula is, in general, false \cite{Chernoff-Memoirs,Chernoff-Converse}.
We implicitly prove a variation of this converse that is quite specific to this setting (although it seems plausible that this proof may be adapted for a robust result in the setting of complex composition operators).
The following facts are necessary at several distinct steps of the proof of our main theorem so they are isolated for easy reference.

\begin{lemma}\label{Tech-lemma}
Let $\{ \mu_{n}: n \in \mathbb{N} \} \subset \mathcal{M}$  satisfy $$\underbrace{\mu_{n} \rhd \mu_{n} \cdots \rhd \mu_{n}}_{k_{n}} \rightarrow \nu_{\rhd}^{m, \gamma,\sigma}$$
weakly.  The following are true:
\begin{enumerate}
\item The family of measures $\{\mu_{n}^{\rhd j}: n\in \mathbb{N}, \ j=1,\ldots, k_{n} \}$ is tight.
\item The family $\set{k_n \Im (F_{\mu_n}(z) - z)}_{n \in \mf{N}}$ is pointwise bounded, and for every $\epsilon > 0$ there exist $\alpha , \beta > 0$ such that
$$k_{n} \Im{\left(F_{\mu_{n}}(z) - \frac{1}{\mu_n(\mf{R})} z \right)} \frac{m^{-1} - 1}{-\log(m)} \leq 2\Im{ \left(F_{\mu_{n}}^{\circ k_{n}}(z) - \frac{1}{\mu_n(\mf{R})^{k_n}} z \right)} \leq \epsilon |z| $$ for large
 $n \in \mathbb{N}$ and $z\in \Gamma_{\alpha,\beta}$.
\end{enumerate}
\end{lemma}

\begin{proof}
We begin with property $(a)$ listed above.
The family $\set{\mu_n^{\rhd k_n}}_{n \in \mf{N}}$ converges to $\nu_{\rhd}^{m, \gamma, \sigma}$ by assumption, and therefore is tight.
It follows that for any $\eps$, there is a Stolz angle $\Gamma_{\alpha, \beta}$, such that $\Im(F_{\mu_n}^{\circ k_n}(z) - \frac{1}{\mu_n(\mf{R})^{k_n}} z) < \eps \abs{z}$ for $z \in \Gamma_{\alpha, \beta}$.
Since $F_{\mu_n}$ increases the imaginary part, it is also true that  $\Im(F_{\mu_n}^{\circ j}(z) - \frac{1}{\mu_n(\mf{R})^j} z) < \eps \abs{z}$ for any $1 \leq j \leq k_n$.

In order to show that the family $\set{\mu_n^{\rhd j} : n \in \mf{N}, 1 \leq j \leq k_n}$ is tight, we assume that this does not hold and obtain a contradiction.
Suppose that there exists a $\delta > 0$ such that for any $K \in \mf{N}$ there are $j(K)$, $n(K)$ with $\mu_{n(K)}^{\rhd j(K)}([-K,K]^c) > \delta$.
We define $$\rho_K = \mu_{n(K)}^{\rhd j(K)}|_{[-K,K]}, \ \ \ \lambda_K = \mu_{n(K)}^{\rhd j(K)} - \rho_K$$ Note that $\rho_K(\mf{R}) \leq \mu_{n(K)}(\mf{R})^{j(K)} - \delta$. It then follows from the Nevanlinna representation of $F_{\rho_K}$ that
\[
\Im (F_{\rho_K}(z)) \geq (\mu_{n(K)}(\mf{R})^{j(K)} - \delta)^{-1} \Im (z),
\]
\[
\Im \left(F_{\rho_K}(z) - \frac{1}{\mu_{n(K)}(\mf{R})^{j(K)}} z \right) \geq \frac{\delta}{(\mu_{n(K)}(\mf{R})^{j(K)} - \delta) \mu_{n(K)}(\mf{R})^{j(K)}} \Im (z).
\]
Also, for any fixed $z$, $G_{\lambda_K}(z) \rightarrow 0$ as $K \rightarrow \infty$. Since $$F_{\mu_{n(K)}^{\rhd j(K)}} = \frac{F_{\rho_K}}{1 + F_{\rho_K} G_{\lambda_K}}$$ it follows that for sufficiently large $K$,
\[
\Im \left(F_{\mu_{n(K)}^{\rhd j(K)}}(z) - \frac{1}{\mu_{n(K)}(\mf{R})^{j(K)}} z \right) \geq \frac{\delta}{\mu_{n(K)}(\mf{R})^{2 j(K)}} \Im (z).
\]
Taking $z$ purely imaginary and $\eps = \frac{\delta}{\mu_{n(K)}(\mf{R})^{2 j(K)}}$, we obtain a contradiction.

In order to address property $(b)$,
we first claim that for $\epsilon > 0$ there exists a Stolz angle $\Gamma_{\alpha,\beta}$ so that
\[
\Im{ \left(F_{\mu_{n}}^{\circ j}(z) - \frac{1}{\mu_n(\mf{R})} F_{\mu_{n}}^{\circ (j-1)}(z) \right)} \geq (1-\epsilon) \Im{ \left(F_{\mu_{n}}(z) - \frac{1}{\mu_n(\mf{R})} z \right)}
\]
for all $n \in \mathbb{N}$, $j=1,\ldots,k_{n}$ and $z \in \Gamma_{\alpha,\beta}$.

Indeed, consider the following chain of equalities and inequalities:
\begin{align*}
& \Im{ \left(F_{\mu_{n}}^{\circ j}(z) - \frac{1}{\mu_n(\mf{R})} F_{\mu_{n}}^{\circ (j-1)}(z) \right)}
\\ &= \Im{\left(\left[ -\gamma_{n} + \frac{1}{\mu_n(\mf{R})} F_{\mu_{n}}^{\circ (j-1)}(z) + \int_{\mathbb{R}}\frac{1+tF_{\mu_{n}}^{j-1}(z)}{t-F_{\mu_{n}}^{\circ (j-1)}(z)}d\sigma_{n}(t) \right] - \frac{1}{\mu_n(\mf{R})} F_{\mu_{n}}^{\circ (j-1)}(z)   \right)}
\\ &= \int_{\mathbb{R}} \frac{\Im{(F_{\mu_{n}}^{\circ (j-1)}(z))}(1+t^{2})}{|t-F_{\mu_{n}}^{\circ (j-1)}(z)|^{2}}d\sigma_{n}(t)
\\ &= \int_{\mathbb{R}} \frac{\Im{(F_{\mu_{n}}^{\circ (j-1)}(z))}(1+t^{2})}{|t-z|^{2}} \frac{|t-z|^{2}}{|t-F_{\mu_{n}}^{\circ (j-1)}(z)|^{2}}d\sigma_{n}(t)
\\ &\geq \int_{\mathbb{R}} \frac{\Im{(z)}(1+t^{2})}{|t-z|^{2}} \frac{|t-z|^{2}}{|t-F_{\mu_{n}}^{\circ (j-1)}(z)|^{2}}d\sigma_{n}(t)
\\ &\geq \inf_{t\in \mathbb{R}} \left\{\frac{|t-z|^{2}}{|t-F_{\mu_{n}}^{\circ (j-1)}(z)|^{2}} \right\}  \int_{\mathbb{R}} \frac{\Im{(z)}(1+t^{2})}{|t-z|^{2}} d\sigma_{n}(t)
\\ &= \inf_{t\in \mathbb{R}} \left\{ \frac{|t-z|^{2}}{|t-F_{\mu_{n}}^{\circ (j-1)}(z)|^{2}} \right\}  \Im{ \left(F_{\mu_{n}}(z) - \frac{1}{\mu_n(\mf{R})} z \right)}
\end{align*}
The inequalities arise because the above integrands are non-negative.  The first inequality is a result of the fact that $F$-transforms increase the imaginary part.

Our claim will follow if we can show that the infimum above is arbitrarily close to $1$ for all $z$ in a sufficiently small Stolz angle.
Indeed, we have shown that $\{\mu_{n}^{\rhd j} \}_{n\in \mathbb{N}, \ j=1,\ldots,k_{n}}$ forms a tight family of measures.  By Lemma~\ref{Lemma:Basic-facts}(c),
this implies that $|F_{\mu_{n}}^{\circ j}(z) - \frac{1}{\mu_n(\mf{R})^j} z| = o(|z|)$ uniformly over $j$ and $n$ for $z$ in a sufficiently small Stolz angle.
The claim follows from simple geometric considerations. Note that tightness also implies that the infimum is finite for every fixed $z$ (that is, $z$ need not lie in the Stolz angle).

Next, observe that we may utilize this claim to attain a bound for $k_{n}\Im{(F_{\mu_{n}}(z) - \frac{1}{\mu_n(\mf{R})} z)}$ on this Stolz angle.
Indeed, if we recall that $\Im{(F_{\mu_{n}}^{\circ j}(z) - \frac{1}{\mu_n(\mf{R})} F_{\mu_{n}}^{\circ (j-1)}(z) )} \geq 0$ for $j=1,\ldots , k_{n}$, a mild telescoping argument implies that
\begin{align*}
k_n \Im{\left(F_{\mu_{n}}(z) - \frac{1}{\mu_n(\mf{R})} z \right)} \frac{m^{-1} - 1}{-\log(m)}
& \approx \Im{\left(F_{\mu_{n}}(z) - \frac{1}{\mu_n(\mf{R})} z \right)} \frac{\frac{1}{\mu_n(\mf{R})^{k_n}} - 1}{\frac{1}{\mu_n(\mf{R})} - 1} \\
& = \Im{\left(F_{\mu_{n}}(z) - \frac{1}{\mu_n(\mf{R})} z \right)} \sum_{j=1}^{k_n} \frac{1}{\mu_n(\mf{R})^{k_n - j}} \\
  & \leq (1-\epsilon)^{-1} \sum_{j=1}^{k_{n}} \Im{ \left(F_{\mu_{n}}^{\circ j}(z) - \frac{1}{\mu_n(\mf{R})} F_{\mu_{n}}^{\circ (j-1)}(z) \right)} \frac{1}{\mu_n(\mf{R})^{k_n - j}}
 \\   & =  (1-\epsilon)^{-1} \Im{ \left(F_{\mu_{n}}^{\circ k_{n}}(z) - \frac{1}{\mu_n(\mf{R})^{k_n}} z \right)}
\end{align*}
for $\eps < 1/2$ and sufficiently large $n$, and where $\frac{m^{-1} - 1}{-\log(m)} = 1$ for $m=1$.  Note that the right hand side of this inequality is uniformly $o(|z|)$ for $z \in \Gamma_{\alpha, \beta}$.  This is a result of the fact that the $F$-transforms of
a uniformly bounded, tight family of measures have this property (Lemma \ref{Lemma:Basic-facts}) where the tightness is a consequence of the fact that the family of measures converges.
This proves the second statement in part $(b)$. Finally, if $z$ is fixed, pointwise finiteness of the infimum and the same argument imply the first statement in part (b).
\end{proof}

\begin{remark}\label{remark}
Part $(a)$ utilizes an approach found in \cite{Williams}. Part $(b)$ in the previous lemma provides an estimate for the finite measures arising from the Nevanlinna representations of the associated $F$-transforms.
Indeed, given that $$F_{\mu_{n}}(z) = - \gamma_{n} + \frac{1}{\mu_n(\mf{R})} z + \int_{\mathbb{R}}\frac{1+tz}{t-z}d\sigma_{n}(t)$$ we have that
$$k_n \sigma_{n}(|t|>y) \leq 2 k_{n}\int_{\mathbb{R}}\frac{1+t^{2}}{t^{2} + y^{2}}d\sigma_{n}(t) = \frac{2 k_{n}}{y} \Im{ \left(F_{\mu_{n}}(iy) - \frac{1}{\mu_n(\mf{R})} iy \right)},$$
and the previous lemma provides us with a bound for the right hand side of the inequality (in the case where the monotone infinitesimal array converges).
This estimate will be used in our proof of the main theorem.
\end{remark}

\begin{proposition}
\label{Prop:Additive}
Let $\set{\mu_n}_{n \in \mf{N}}$ be a family of sub-probability measures on $\mf{R}$. Then
\[
\mu_n^{\uplus k_n} \rightarrow \nu_{\uplus}^{m, \gamma, \sigma}
\]
weakly if and only if
\[
\mu_n^{\rhd k_n} \rightarrow \nu_{\rhd}^{m, \gamma, \sigma}
\]
weakly.
\end{proposition}

\begin{proof}
As a first step, we extend the result in Proposition \ref{Prop:Variance} to full generality. Assume that $\{\mu_{n} \}_{n \in \mathbb{N}}$ satisfies
$\mu_{n} \uplus \cdots \uplus \mu_{n} \rightarrow \nu^{m, \gamma,\sigma}_{\uplus}$
where we have no moment assumptions on any of the relevant probability measures.

Consider the functions
$$ -E_{\mu_{n}}(z)  =  F_{\mu_{n}}(z) - \frac{1}{\mu_n(\mf{R})} z = - \gamma_{n} + \int^{\infty}_{-\infty} \frac{1+tz}{t-z}d\sigma_{n}(t)$$
where the function on the right hand side is the Nevanlinna representation
and recall that our hypothesis is equivalent to
$$- k_{n}E_{\mu_{n}}(z) \rightarrow - \gamma + \int_{-\infty}^{\infty} \frac{1+tz}{t-z }d\sigma (t) = - E_{\nu^{m, \gamma,\sigma}_{\uplus}}(z)$$where the convergence is uniform on compact subsets of $\mathbb{C}^{+}$.
It follows from this fact that $k_{n}\gamma_{n} \rightarrow \gamma$
and $k_{n} \sigma_{n} \rightarrow \sigma$ where the latter is with respect to the weak topology.

We define a new family of measures $\{ \mu_{n,N} : n,N \in \mathbb{N} \}$
implicitly through the equation
$$- E_{\mu_{n,N}}(z) = F_{\mu_{n,N}}(z) - \frac{1}{\mu_{n}(\mf{R})} z  = - \gamma_{n} + \int^{N + \epsilon(n,N)}_{-(N + \delta(n,N))} \frac{1+tz}{t-z}d\sigma_{n}(t) $$
where the (small) numbers $\epsilon(n,N), \delta(n,N)$ are chosen so that
$$k_{n} \sigma_{n}|_{[-N - \delta(n,N) , N + \epsilon(n,N)]} \rightarrow \sigma|_{[-N,N]} = \widetilde{\sigma}_{N}$$
(mass may converge to $N$ so this slight correction is required).
It follows that, for each $N$, the functions $k_{n}E_{\mu_{n,N}}$ converge to $E_{\nu_{\uplus}^{m, \gamma,\widetilde{\sigma}_{N}}}$ uniformly on compact subsets of $\mathbb{C}^{+}$. Since also $\mu_{n, N}(\mf{R})^{k_n} = \mu_n(\mf{R})^{k_n} \rightarrow m$, this is equivalent to $\mu_{n,N} \uplus \cdots \uplus \mu_{n,N} \rightarrow \nu_{\uplus}^{m, \gamma, \widetilde{\sigma}_{N}}$
in the weak topology.  The reader should note that these measures have support contained in $[-N - \delta(N) , N + \epsilon(N)]$ where $\delta(N) = \sup_{n \in \mathbb{N}}\delta(n,N)$
$\epsilon(N) = \sup_{n \in \mathbb{N}}\epsilon(n,N)$ which we may assume is as close to $0$ as we would like.  Thus, the measures are compactly supported in a uniform sense
so that the hypotheses of Proposition \ref{Prop:Variance} are satisfied (as this implies the requisite uniform bound on the variance).

Consider the following inequality:
$$|F_{\mu_{n}^{\rhd {k_{n}}}}(z) - F_{\nu^{m, \gamma,\sigma}_{\rhd}}(z)|$$
$$\leq |F_{\mu_{n}}^{\circ k_{n}}(z) - F_{\mu_{n,N}}^{\circ k_{n}}(z)|  +
|F_{\mu_{n,N}^{\rhd {k_{n}}}}(z) - F_{\nu^{m, \gamma,\widetilde{\sigma}_{N}}_{\rhd}}(z)|
 + |F_{\nu^{m, \gamma,\widetilde{\sigma}_{N}}_{\rhd}}(z) - F_{\nu^{m, \gamma,\sigma}_{\rhd}}(z)|$$
where we shall refer to the terms on the right hand side of the inequality as $(1)$, $(2)$ and $(3)$, respectively.
We claim that we may make each of these terms arbitrarily small with the proper choice of $N$ and $n$ large enough.

We begin by bounding $(3)$.  Choose $K \subset \mf{C}_1^+$ compact and $\epsilon > 0$.
$F_{\nu^{m, \gamma,\sigma}_{\rhd}}(z)$ is the solution at time $1$ of the initial value problem
\[
\frac{\partial_t F_t(z)}{\partial t} + E_{\nu^{m, \gamma,\sigma}_{\uplus}}(F_t(z)) + \log(m) F_t(z) = 0 \ ; \ \ \ \ F_{0}(z) = z
\]
in $\mathbb{C}^{+}$, while $F_{\nu^{m, \gamma,\widetilde{\sigma}_{N}}_{\rhd}}$ is the solution of the corresponding system involving $E_{\nu^{m, \gamma,\widetilde{\sigma}_N}_{\uplus}}$.
By Lemma~\ref{Lemma:Distance} there exists $N_{0} \in \mathbb{N}$ such that term $(3) < \epsilon$ on $K$ for $N \geq N_0$.

In order to control terms $(1)$ and $(2)$, note that $\nu^{m, \gamma,\tilde{\sigma}_{N}}_{\rhd} \rightarrow \nu^{m, \gamma,\sigma}_{\rhd}$ as $N\uparrow \infty$ (this follows from Lemma \ref{Lemma:Distance}).
Choose a family $\{U_{N} \}_{N\in \mathbb{N}}$ such that
\begin{enumerate}
 \item $U_{N}$ is a weak neighborhood of $\nu^{m, \gamma,\sigma}_{\rhd}$
\item $\nu^{m, \gamma,\tilde{\sigma}_{N}}_{\rhd} \in U_{N}$
\item $U_{N+1} \subset U_{N}$
\item $\cap_{N=1}^{\infty}U_{N} = \{\nu^{m, \gamma,\sigma}_{\rhd} \}$
\end{enumerate}
Invoking Proposition \ref{Prop:Variance}, for every $N \in \mathbb{N}$, there exists $n(N) \in \mathbb{N}$ such that $(2) < \epsilon$
for $n \geq n(N)$.  We may further assume that $n(N)$ is chosen large enough so that $\mu_{n,N}^{\rhd k_{n}} \in U_{N}$ (the Proposition implies that these measures converge to $\nu^{m, \gamma,\tilde{\sigma}_{N}}_{\rhd} \in U_{N}$ as $n\uparrow \infty$).
Since tightness is equivalent to sequential precompactness, we have that
$\{\mu_{n,N}^{\rhd k_{n}} : \ N \in \mathbb{N}, \ n \geq n(N) \}$ forms a tight family (our neighborhoods $U_{N}$ were chosen for this purpose).
 By the same argument as in Lemma \ref{Tech-lemma}, we have that $\{ \mu_{n,N}^{\rhd j} : \ N \in \mathbb{N}, \ n \geq n(N) , \ j=1, \ldots , k_{n} \}$ forms a tight family, so that
$$C = \bigcup_{N\in \mathbb{N}} \bigcup_{n \geq n(N)} \bigcup_{j=1}^{k_{n}} F_{\mu_{n,N}}^{\circ j}(K) $$ has compact closure.  Let $M_{0}$ be the upper bound on the magnitude of the derivative of the family of
functions $\{ F_{\mu_{n}}^{\circ j} : n \in \mathbb{N}, j=1,\ldots , k_{n} \}$ on this set $C$ and $M_{1}$ denote the upper bound on the magnitude of the elements in $C$.
The upper bound $M_{0}$ exists since this family of $F$-transforms arises from a uniformly bounded, tight family of measures.  Since this implies sequential precompactness on compact sets, this family is normal, so that such an upper bound
exists for any compact set $C$.

Since $M_{0}$ and $M_{1}$ do not depend on our choice of $N$,
 we may refine our choice of $N_{0}$ so that $$ \frac{\Im{(F_{\nu_\uplus^{m, \gamma, \sigma}}(iN) - \frac{1}{m} iN)}}{N} < \frac{\epsilon}{4M_{0}M_{1}} , \ \ \ \ \ \sup_{z \in C , |t| > N} \frac{1 + tz}{t - z} < (1 + \epsilon)M_{1}$$
for all $N \geq N_{0}$ (the $t$ are real numbers).  The first estimate follows from the asymptotics of the $F$-transform.  These inequalities will play a role in bounding term $(1)$

Now fix $N \geq N_{0}$ and assume that $n \geq n(N)$ so that terms $(2)$ and $(3)$ are both bounded by $\epsilon$.
For any $z \in K$, we have the following inequality involving term $(1)$:

\begin{align*}
(1)  & = |F_{\mu_{n}}^{\circ k_{n}}(z) - F_{\mu_{n,N}}^{\circ k_{n}}(z)| \\
& \leq |F_{\mu_{n}}^{\circ k_{n} - 1}(F_{\mu_{n}}(z)) - F_{\mu_{n}}^{\circ (k_{n}-1)}(F_{\mu_{n,N}}(z))| + |F_{\mu_{n}}^{\circ (k_{n} - 1)}(F_{\mu_{n,N}}(z)) - F_{\mu_{n,N}}^{\circ (k_{n}-1)}(F_{\mu_{n,N}}(z))| \\
& \leq M_{0}|F_{\mu_{n}}(z) - F_{\mu_{n,N}}(z)| + |F_{\mu_{n}}^{\circ (k_{n} - 2)}(F_{\mu_{n}}(F_{\mu_{n,N}}(z))) - F_{\mu_{n}}^{\circ (k_{n} - 2)}(F_{\mu_{n,N}}(F_{\mu_{n,N}}(z))) | \\ & \ \ \ \ +
|F_{\mu_{n}}^{\circ (k_{n} - 2)}(F_{\mu_{n,N}}(F_{\mu_{n,N}}(z))) - F_{\mu_{n,N}}^{\circ (k_{n} - 2)}(F_{\mu_{n,N}}(F_{\mu_{n,N}}(z)))|
\end{align*}

Continuing in this way, we get the estimate
$$ (1) \leq M_{0} \sum_{j=0}^{k_{n}-1} |F_{\mu_{n,N}} \circ F_{\mu_{n,N}}^{\circ (k_{n} - j - 1)}(z) - F_{\mu_{n}} \circ F_{\mu_{n,N}}^{\circ (k_{n} - j - 1)}(z)|$$

For the key step in the estimate, observe that, for $z \in C$
\[
\abs{F_{\mu_{n,N}}(z) - F_{\mu_{n}}(z)} \leq \abs{\int_{\mathbb{R}\setminus [-N,N]} \frac{1+t z}{t- z}d\sigma_{n}(t)} \leq \sup_{z \in C, \ \abs{t} > N} \abs{\frac{1+t z}{t- z}} \sigma_n(\abs{t} > N).
\]
and note that we have already shown that the supremum has a bound of $(1 + \epsilon)M_{1}$.  Now, recalling Remark~\ref{remark} and the fact that $N \geq N_{0}$,
\[
\begin{split}
k_n \sigma_{n}(|t| > N)
& \leq \frac{2 k_n \Im{(F_{\mu_{n}}(iN) - \frac{1}{\mu_n(\mf{R})} iN)}}{N} \\
& =  \frac{2\Im{(F_{\mu_{n}^{\uplus_{k_{n}}}}(iN) - \frac{1}{\mu_n(\mf{R})^{k_n}} iN)}}{N}
\leq \frac{4 \Im(F_{\nu_\uplus^{m, \gamma, \sigma}}(i N) - \frac{1}{m} i N)}{N}.
\end{split}
\]
(Note that the last of these inequalities simply follows from the fact that $\mu_{n}^{\uplus_{k_{n}}} \rightarrow \nu_\uplus^{m, \gamma, \sigma}$ and a fundamental fact about the asymptotics of the F-transforms of convergent families of measures found in \cite{BV93}.
As such, it may be necessary to choose our $n$ larger, but this does not create any new dependence since we have fixed our $N \geq N_{0}$.)

Thus, our estimate for $(1)$ becomes
$$ (1) \leq M_{0} \sum_{j=0}^{k_{n} - 1} \frac{4 \abs{z_{j}} \Im{(F_{\nu_\uplus^{m, \gamma, \sigma}}(iN) - \frac{1}{m} iN)}}{Nk_{n}} = M_{0} \frac{4 \Im{(F_{\nu_\uplus^{\gamma, \sigma}}(iN) - \frac{1}{m} iN)}}{N} \sum_{j=0}^{k_{n} - 1} \frac{\abs{z_{j}}}{k_{n}},$$
where these $$z_{j} = F_{\mu_{n,N}}^{\circ (k_{n} - j - 1)}(z) \in C , \ \ \ n \geq n(N) , \ \ \ N \geq N_0$$ Combining our estimates, we have that
$$(1) \leq (4M_{0}(1+\epsilon)M_{1})\left(\frac{\epsilon}{4M_{0}M_{1}} \right) = \epsilon (1 + \epsilon)$$
Thus, $$|F_{\mu_{n}^{\rhd {k_{n}}}}(z) - F_{\nu^{m, \gamma,\sigma}_{\rhd}}(z)|< (3 + \epsilon)\epsilon$$  for $z \in C$ and $n \geq n(N)$. This implies
that these functions converge uniformly on compact sets which is equivalent to the fact that  $\mu_{n}^{\rhd {k_{n}}} \rightarrow \nu^{m, \gamma,\sigma}_{\rhd}$ weakly.
This completes the proof of the forward direction.

We now assume that $\mu_{n} \rhd \mu_{n} \rhd \cdots \rhd \mu_{n} \rightarrow \nu_{\rhd}^{m, \gamma,\sigma}$ and claim that this implies
$\mu_{n} \uplus \mu_{n} \uplus \cdots \uplus \mu_{n} \rightarrow \nu_{\uplus}^{m, \gamma,\sigma}$.
To see this, note that Lemma \ref{Tech-lemma} implies that, for fixed large $y > 0$ and sufficiently large $n$, we have
\begin{align*}
& k_{n}\sigma_{n}(\{|t| > y \})  \leq 2k_{n}\int_{\mathbb{R}} \frac{1+t^{2}}{y^{2} + t^{2}}d\sigma_{n}(t)
\\ & = \frac{2k_{n}}{y}\Im{(F_{\mu_{n}}(iy) - \frac{1}{\mu_n(\mf{R})} iy)}
 \leq \frac{-\log(m)}{m^{-1} - 1} \frac{4\Im{(F_{\mu_{n}}^{\circ k_{n}}(iy) - \frac{1}{\mu_n(\mf{R})^{k_n}} iy)}}{y}.
\end{align*}
Using Lemma~\ref{Lemma:Basic-facts}(c) again, the right hand side of the inequality goes to $0$ as $y\uparrow \infty$ (since this is a tight family). Also,
\[
k_n \sigma_n(\mf{R}) = k_n \Im \left(F_{\mu_n}(i) - \frac{1}{\mu_n(\mf{R})} i \right),
\]
so the first statement in Lemma~\ref{Tech-lemma}(b) implies that $k_{n}\sigma_{n}(\mathbb{R})$ is bounded over $n$. Therefore this estimate implies tightness of the family of measures $k_{n}\sigma_{n}$.

Now, let $\sigma'$ denote a weak cluster point for this family of finite measures.  We claim that $k_{n} \gamma_{n}$ is bounded along the relevant subsequence.
Indeed, note that $$\Re(F_{\mu_{n}}^{\circ k_{n}}(i)) = - k_{n} \gamma_{n} + \sum_{j=0}^{k_{n}-1} \int_{\mathbb{R}} \frac{(t - \Re(F_{\mu_{n}}^{\circ j} (i)))(1 + t\Re(F_{\mu_{n}}^{\circ j} (i))) -
 \Im(F_{\mu_{n}}^{\circ j}(i))^{2}}{|t - F_{\mu_{n}}^{\circ j}(i)|^{2}}d\sigma_{n}(t)$$
The left hand side of the equation is convergent by assumption. As we have shown that
\[
\{ \mu_{n}^{\rhd j} \}_{n \in \mathbb{N}, j=1,\ldots , k_{n}}
\]
is a tight family under these assumptions, the magnitude of the integrands on the right hand side of the equation have a uniform upper bound of $c$.
Thus, we have the inequality,
\[
\abs{\Re(F_{\mu_{n}}^{\circ k_{n}}(i)) + k_{n} \gamma_{n}} \leq  k_{n}\sigma_{n}(\mathbb{R})c
\]
which implies that $k_{n}\gamma_{n}$
is a bounded sequence.  Thus, we may additionally assume that $k_{n}\gamma_{n} \rightarrow \gamma'$ along this subsequence.

  Consider the function $$F^{\gamma',\sigma'}(z) = \frac{1}{m} z - \gamma' + \int_{\mathbb{R}}\frac{1+tz}{t-z}d\sigma'(t).$$
Then, along an appropriate subsequence, we have that $k_{n} (F_{\mu_{n}}(z) - \frac{1}{\mu_n(\mf{R})} z) \rightarrow F^{\gamma',\sigma'}(z) - \frac{1}{m} z$ uniformly on compact subsets of $\mathbb{C}^{+}$. Along with $\mu_n(\mf{R})^{k_n} \rightarrow m$, this implies that $\mu_{n}^{\uplus_{k_{n}}} \rightarrow \nu_{\uplus}^{m, \gamma',\sigma'}$ along an appropriate subsequence.
But we have just shown that this fact implies that $\mu_{n}^{\rhd k_{n}} \rightarrow \nu_{\rhd}^{m, \gamma',\sigma'}$ along this subsequence.
Since we are assuming that  $\mu_{n}^{\rhd k_{n}} \rightarrow \nu_{\rhd}^{m, \gamma,\sigma}$ we may conclude that $\gamma = \gamma'$ and $\sigma = \sigma'$.
This completes our proof.
\end{proof}

\begin{proof}[Proof of Theorem~\ref{Theorem:Main}]
Equivalence of parts (c) and (d) follows by applying Proposition~\ref{Prop:Additive} to probability measures. The remaining equivalences were proved in \cite{BerPatDomains}.
\end{proof}

\section{Limit theorems for multiplicative monotone convolution on $\mf{T}$}

\subsection{Preliminaries}

Let $\mu$ denote a probability measure on the unit circle $\mf{T}$.  We define the transforms
\[
\psi_\mu(z) = \int_{\mf{T}} \frac{z \zeta}{1 - z \zeta} \,d\mu(\zeta), \quad \eta_\mu(z) = \frac{\psi_\mu(z)}{1 + \psi_{\mu}(z)}.
\]
Note that the mean of $\mu$ is
\begin{equation}
\label{Eq:Mean}
\int_{\mathbb{T}}\zeta d\mu(\zeta) = \lim_{z \rightarrow 0} \frac{\eta_\mu(z)}{z} = \eta'(0).
\end{equation}
We will always assume that this quantity is non-zero, in which case $\eta_{\mu}^{-1}$ is defined in a neighborhood of zero and we may define a new transform
$$\Sigma_{\mu}(z) = \frac{\eta_{\mu}^{-1}(z)}{z}$$
It is immediate from the definition of $\eta_\mu$ that it takes the unit disk $\mf{D}$ to itself and $\eta_\mu(0) = 0$, so that in fact for $z \in \mf{D}$, $\abs{\eta_\mu(z)} \leq \abs{z}$. This fact is necessary in what follows as we will treat these transforms as composition operators on certain spaces of functions on
\[
\mathbb{D}_{1/2} = \set{z \in \mf{C} : \abs{z} < \frac{1}{2}}.
\]

By taking products of random variables that are freely, Boolean and monotonically independent, we may develop multiplicative forms of convolution.
We will forgo the operator algebraic definition of the convolution operations and refer to \cite{Voiculescu_Mult}, \cite{Franz_Bool_Circle} and \cite{Bercovici-Multiplicative-monotonic} for the theory relevant to
free, Boolean and monotone convolution, respectively.
Instead, given  probability measures $\mu$ and $\nu$, we define the free, Boolean and monotone multiplicative convolution operations (in symbols $\mu \boxtimes \nu$, $\mu \utimes \nu$ and $\mu \circlearrowright \nu$)
through their transforms.  That is, we define binary operations on the space of probability measures on $\mathbb{T}$ implicitly through their transforms as follows:
$$\Sigma_{\mu \boxtimes \nu}(z) = \Sigma_{\mu}(z) \Sigma_{\nu}(z), \ \ \ \frac{\eta_{\mu \sutimes \nu}(z)}{z} = \frac{\eta_{\mu}(z)}{z} \frac{\eta_{\nu}(z)}{z},  \ \ \ \eta_{\mu \circlearrowright \nu}(z) = \eta_{\mu} \circ \eta_{\nu}(z) $$
It follows from equation~\eqref{Eq:Mean} that
\[
\int_{\mathbb{T}}\zeta d(\mu \utimes \nu)(\zeta) = \int_{\mathbb{T}}\zeta d\mu(\zeta) \int_{\mathbb{T}}\zeta d\nu(\zeta) = \int_{\mathbb{T}}\zeta d(\mu \circlearrowright \nu)(\zeta).
\]
Note that there is an analogous version of multiplicative convolution for classical independence.  This binary operation is represented by the symbol $\circledast$.
We will not discuss this type of convolution directly although it figures in our results.

According to \cite{Bercovici-Multiplicative-monotonic} (see also \cite{Berkson-Porta}), the $\eta$-transforms of a multiplicative monotone convolution semigroup satisfy an equation of the form
\[
\frac{d \eta_{\mu_t}(z)}{d t} = A(\eta_{\mu_t}(z)),
\]
where the generator $A$ of the semigroup is a general function of the form $A(z) = z B(z)$, where $B$ is analytic in $\mf{D}$ and $\Re (B(z)) \leq 0$, in other words
\[
B(z) = i \beta - \int_{\mf{T}} \frac{1 + \zeta z}{1 - \zeta z} \,d\sigma(\zeta).
\]

According to \cite{Wang-Boolean}, we may identify the classes of $\circledast$, $\boxtimes$ and $\utimes$ infinitely divisible Borel probability measures on $\mathbb{T}$ with $\gamma \in \mf{T}$ and $\sigma$ a finite Borel measure on $\mf{T}$
\[
(\mc{F} \nu_\circledast^{\gamma, \sigma})(p) = \gamma^p \exp \left( \int_{\mf{T}} \frac{\zeta^p - 1 - i p \Im (\zeta)}{1 - \Re (\zeta)} \,d\sigma(\zeta) \right), \quad p \in \mf{Z},
\]
\[
\Sigma_{\nu_\boxtimes^{\gamma, \sigma}}(z) = \gamma \exp \left( \int_{\mf{T}} \frac{1 + \zeta z}{1 - \zeta z} \,d\sigma(\zeta) \right), \quad z \in \mf{D},
\]
\[
\eta_{\nu_{\sutimes}^{\gamma, \sigma}}(z) = \gamma z \exp \left( - \int_{\mf{T}} \frac{1 + \zeta z}{1 - \zeta z} \,d\sigma(\zeta) \right), \quad z \in \mf{D}.
\]

\begin{definition}
For $\beta \in \mf{R}$ and $\sigma$ as above, denote
\[
A^{\beta, \sigma}(z) = z \left(i \beta - \int_{\mf{T}} \frac{1 + \zeta z}{1 - \zeta z} \,d\sigma(\zeta) \right),  \quad z \in \mf{D}.
\]
Note that if $\gamma = e^{i \beta}$, then
\[
\exp \left( \frac{1}{z} A^{\beta, \sigma}(z) \right) = \frac{1}{z} \eta_{\nu_{\sutimes}^{\gamma, \sigma}}(z).
\]
Let $\set{\nu_t : t \geq 0}$ be the multiplicative monotone convolution semigroup $A^{\beta, \sigma}$ generates, and denote $\nu_{\circlearrowright}^{\beta, \sigma} = \nu_1$.
\end{definition}

\begin{remark}
\label{Remark:Model}
According to \cite{Guryainov} (see \cite{CM95} for the background), any multiplicative monotone convolution semigroup has the form
\[
\eta_{\nu_t}(z) = u^{-1} (r^t e^{i \theta t} u(z))
\]
for $0 < r \leq 1$ and $\theta \in \mf{R}$. Here $r^t e^{i \theta t}$ is the mean of $\nu_t$, and the generator of the semigroup is
\[
A(z) = (\log r + i \theta) \frac{u(z)}{u'(z)},
\]
which thus specifies the conditions on the function $u$. $\nu_1$ determines $r$ uniquely, $u$ up to a multiplicative constant, and $\theta$ up to an additive multiple of $2 \pi$. It follows that if $\nu^{\beta_1, \sigma}_\circlearrowright = \nu^{\beta_2, \sigma}_\circlearrowright$, then
\[
2 \pi i k \frac{u(z)}{u'(z)} = z i (\beta_1 - \beta_2),
\]
and the two generators are constant multiples of $z$. Such functions generate monotone convolution semigroups of delta measures. In all other cases, $\nu^{\beta_1, \sigma}_\circlearrowright \neq \nu^{\beta_2, \sigma}_\circlearrowright$. In particular, the reasoning behind the non-uniqueness arguments in the last section of \cite{Bercovici-Multiplicative-monotonic} does not hold. Nevertheless, as Hari Bercovici has pointed out to us, Theorem 4.4 and Proposition 4.5 (and their proofs) in that section are correct: any $\circlearrowright$-infinitely divisible distribution can be included in a $\circlearrowright$-convolution semigroup, and specifying the means determines the $\circlearrowright$-square root uniquely.
\end{remark}

\begin{lemma}
Both $\nu_{\sutimes}^{\gamma, \sigma}$ and $\nu_{\circlearrowright}^{\beta, \sigma}$, where $\gamma = e^{i \beta}$, have mean $\gamma e^{-\sigma(\mf{T})}$.
\end{lemma}

\begin{proof}
The first statement is clear directly from the formula. For the second, since
\[
\frac{d \eta_{\nu_t}(z)}{d t} = A^{\beta, \sigma}(\eta_{\nu_t}(z)),
\]
we have
\[
\frac{d \eta_{\nu_t}'(0)}{d t} = (A^{\beta, \sigma})'(\eta_{\nu_t}(0)) \eta_{\nu_t}'(0)
= (A^{\beta, \sigma})'(0) \eta_{\nu_t}'(0)
= (i \beta - \sigma(\mf{T})) \eta_{\nu_t}'(0).
\]
It follows that
\[
\eta_{\nu_t}'(0) = \exp((i \beta - \sigma(\mf{T}))t) = e^{i \beta t} e^{-\sigma(\mf{T}) t}.
\]
The result follows.
\end{proof}

\subsection{Main Results}

The following theorem is a restriction of the results from \cite{Bercovici-Wang-Multiplicative} and \cite{Wang-Boolean} from infinitesimal arrays to sequences of measures. See these references also for results concerning the classical multiplicative convolution $\circledast$, as well as the special case of convergence to the Haar measure.

\begin{theorem}
Fix a finite positive Borel measure $\sigma$ on $\mathbb{T}$, a complex number $\gamma \in \mf{T}$, a sequence of probability measures $\set{\mu_{n}}_{n \in \mf{N}}$ on $\mf{T}$ converging to $\delta_1$ weakly, and a sequence of positive integers
$k_{1} < k_{2} < \cdots $ The following assertions are equivalent:
\begin{enumerate}
 \item
The sequence $\underbrace{\mu_{n} \boxtimes \mu_{n} \boxtimes \cdots \boxtimes \mu_{n}}_{k_{n}}$ converges weakly to $\nu_{\boxtimes}^{\gamma,\sigma}$;
 \item
The sequence $\underbrace{\mu_{n} \utimes \mu_{n} \utimes \cdots \utimes \mu_{n}}_{k_{n}}$ converges weakly to $\nu_{\sutimes}^{\gamma,\sigma}$;
\item
\[
k_n (1 - \Re (\zeta)) \,d\mu_n(\zeta) \rightarrow \sigma
\]
and
\[
\exp \left(i k_n \int_{\mf{T}} \Im (\zeta) \,d\mu_n(\zeta) \right) \rightarrow \gamma.
\]
\end{enumerate}
\end{theorem}

\begin{remark}
A straightforward inclusion of the multiplicative monotone convolution into the preceding theorem does not hold. Indeed, let $\eta(z) = u^{-1}(r e^{i \theta} u(z))$, where $0 < r < 1$. Conditions on $u$ for $\eta$ to be an $\eta$-transform are known, see Remark~\ref{Remark:Model} (for example, we could take $u(z) = 1 - z^{-(k-1)}$ and $r e^{i \theta} = e^{\alpha (k-1)}$). Then
\[
\eta^{\circ (1/k_n)}(z) = u^{-1}(r^{1/k_n} e^{i \theta/k_n} e^{2 \pi i \ell_n/k_n} u(z)).
\]
For large $k_n$, and assuming that $\ell_n/k_n \rightarrow 0$,
\[
\eta^{\circ (1/k_n)}(z) \approx z + \left( r^{1/k_n} e^{i \theta/k_n} e^{2 \pi i \ell_n/k_n} - 1 \right) \frac{u(z)}{u'(z)}.
\]
Therefore
\[
\left( \frac{1}{z} \eta^{\circ (1/k_n)}(z) \right)^{k_n} \approx \exp \left( k_n \left( r^{1/k_n} e^{i \theta/k_n} e^{2 \pi i \ell_n/k_n} - 1 \right) \frac{u(z)}{z u'(z)} \right)
\approx \exp \left( (i \theta + 2 \pi i \ell_n) \frac{u(z)}{z u'(z)} \right),
\]
which depends on the choice of $\ell_n$. So the $k_n$'th Boolean power of the $k_n$'s monotone root need not have a limit.
\end{remark}

\begin{proposition}\label{Multiplicative_Proposition}
Suppose the sequence $\underbrace{\mu_{n} \utimes \mu_{n} \utimes \cdots \utimes \mu_{n}}_{k_{n}}$ converges weakly to $\nu_{\sutimes}^{\gamma,\sigma}$. Then for any $\beta \in \mf{R}$ with $\gamma = e^{i \beta}$, there exist $\lambda_n \in \mf{T}$, $\lambda_n^{k_n} = 1$ such that for $\tilde{\mu}_n = \delta_{\lambda_n} \circlearrowright \mu_{n}$,
the sequence $\underbrace{\tilde{\mu}_{n} \circlearrowright \tilde{\mu}_{n} \circlearrowright \cdots \circlearrowright \tilde{\mu}_{n}}_{k_{n}}$ converges weakly to $\nu_{\circlearrowright}^{\beta,\sigma}$
\end{proposition}

\begin{proof}
To prove this proposition, first note that our hypotheses imply that
\[
\left( \frac{1}{z} \eta_{\mu_n}(z) \right)^{k_n}
\rightarrow \frac{1}{z} \eta^{\gamma, \sigma}(z)
= \gamma \exp \left( - \int_{\mf{T}} \frac{1 + \zeta z}{1 - \zeta z} \,d\sigma(\zeta) \right).
\]
and $\mu_n \rightarrow \delta_1$ weakly so that $\frac{1}{z} \eta_{\mu_n}(z) \rightarrow 1$ on $\mf{D}_{1/2}$. We may take the logarithm of both sides of the above limit,
but there is ambiguity as to the branches of the logarithm.  Fix $\beta$ with $e^{i \beta} = \gamma$. We may conclude that there exists a sequence $\{ \ell_{n} \}_{n\uparrow \infty}$ of integers so that
\begin{equation}
\label{Eq:Rotation}
\left| k_{n}\left[ \log{\left( \frac{1}{z} \eta_{\mu_n}(z)  \right)} + \frac{2 \pi i \ell_{n}}{k_{n}} \right] - \left(i \beta - \int_{\mf{T}} \frac{1 + \zeta z}{1 - \zeta z} \right) \right|  \rightarrow 0
\end{equation}
as $n\uparrow \infty$ where we have fixed the branch of the logarithm with $\Im{(\log(1))} = 0$.
The fact that $ \log{\left( \frac{1}{z} \eta_{\mu_n}(z)  \right)} \rightarrow 0$ implies that the same must be true of $\ell_{n}/k_{n}$.

With these considerations in mind, we introduce the following correction.  Let $\lambda_{n} = e^{2\pi i \ell_{n} / k_{n}}$.  Define $\tilde{\mu}_{n} = \delta_{\lambda_{n}} \circlearrowright \mu_{n} $.
Observe that

$$\left| k_{n} \log{\left( \frac{1}{z} \eta_{\tilde{\mu}_n}(z)  \right)}  - \left(i \beta - \int_{\mf{T}} \frac{1 + \zeta z}{1 - \zeta z} \right)  \right| =
 \left| k_{n} \log{\left( \frac{1}{z} \eta_{(z)}  \right)} + 2\pi \ell_{n}  - \left(i \beta - \int_{\mf{T}} \frac{1 + \zeta z}{1 - \zeta z} \right)  \right|$$
which converges to $0$ uniformly over $\mathbb{D}_{1/2}$.

We next claim that, as $n \uparrow \infty$, $$k_{n} \left| \log \left( \frac{1}{z} \eta_{\tilde{\mu}_n}(z) \right) - \left( \frac{1}{z} \eta_{\tilde{\mu}_n}(z) - 1 \right) \right| \rightarrow 0$$
Indeed, since $\frac{1}{z} \eta_{\tilde{\mu}_n}(z) \rightarrow 1$ (this holds since this limit is true for $\mu_{n}$ and $\lambda_{n} \rightarrow 1$) we may take a series expansion for the logarithm centered at $1$ so that our quantity becomes
$$k_{n} \left|\sum_{p=1}^{\infty} \frac{(1-\frac{1}{z} \eta_{\tilde{\mu}_{n}}(z))^{p} }{p} - \left(1- \frac{1}{z} \eta_{\tilde{\mu}_{n}}(z) \right) \right| = k_{n} \left|\sum_{p=2}^{\infty} \frac{(1-\frac{1}{z} \eta_{\tilde{\mu}_{n}}(z))^{p} }{p} \right| $$
Now, observe that $k_{n}|(1-\frac{1}{z} \eta_{\tilde{\mu}_{n}}(z))|$ is bounded.  Indeed, appealing to the series expansion of the logarithm;
$$|k_{n}\log(\frac{1}{z} \eta_{\tilde{\mu}_{n}}(z))| = |k_{n}(1-\frac{1}{z} \eta_{\tilde{\mu}_{n}}(z)) \sum_{j=1}^{\infty} (1-\frac{1}{z} \eta_{\tilde{\mu}_{n}}(z))^{j-1}/j | \geq |k_{n}(1-\frac{1}{z} \eta_{\tilde{\mu}_{n}}(z))|/2 $$
for n large enough.  Since the left hand side is bounded (it converges to $i \beta - \int_{\mf{T}} \frac{1 + \zeta z}{1 - \zeta z}$ by the previous paragraphs calculation), we have that $k_{n}|(1-\frac{1}{z} \eta_{\tilde{\mu}_{n}}(z))|$ is bounded.
 These facts imply our claim.

Thus, we may conclude that
\[
\lim_{n \uparrow \infty} k_n \left( \frac{1}{z} \eta_{\tilde{\mu}_n}(z) - 1 \right)
= \lim_{n \uparrow \infty} k_n \log \left( \frac{1}{z} \eta_{\tilde{\mu}_{n}}(z) \right)
\rightarrow i \beta - \int_{\mf{T}} \frac{1 + \zeta z}{1 - \zeta z} \,d\sigma(\zeta)
\]
and finally
\[
k_n \left( \eta_{\tilde{\mu}_n}(z) - z \right) \rightarrow A^{\beta, \sigma}(z).
\]

Let
\[
\mc{D} = \set{\psi_{\rho}: \rho \text{ finite Borel measure on } \mf{T}}.
\]
$\mc{D}$ is invariant under composition operators by functions $\eta_\mu$, and these operators are all contractions on it. Let $\mc{A}$ be the completion of $\mc{D}$ with respect to the uniform norm on $\mf{D}_{1/2}$, which we denote by $\norm{\cdot}_\infty$. Then each composition operator by $\eta_\mu$ as above is a contraction on $\mc{A}$. For $h(z) = \psi_\rho \in \mc{D}$,
\[
\begin{split}
& \norm{k_n(h \circ \eta_{\tilde{\mu}_{n}} - h) - A^{\beta, \sigma} h'}_\infty \\
&\qquad = \sup_{z \in \mf{D}_{1/2}} \abs{\int_{\mf{T}} \left(k_n \left( \frac{\eta_{\tilde{\mu}_{n}}(z) \zeta}{1 - \eta_{\tilde{\mu}_{n}}(z) \zeta} - \frac{z \zeta}{1 - z \zeta} \right) - \frac{A^{\beta, \sigma}(z) \zeta}{(1 - z \zeta)^2} \right) \,d\rho(\zeta)} \\
&\qquad = \sup_{z \in \mf{D}_{1/2}} \abs{\int_{\mf{T}} \left(\frac{(k_n(\eta_{\tilde{\mu}_{n}}(z) - z) - A^{\beta, \sigma}(z)) \zeta}{(1 - \eta_{\tilde{\mu}_{n}}(z) \zeta) (1 - z \zeta)} + \frac{(\eta_{\tilde{\mu}_{n}}(z) - z) A^{\beta, \sigma}(z) \zeta}{(1 - \eta_{\tilde{\mu}_{n}}(z) \zeta)(1 - z \zeta)^2} \right) \,d\rho(\zeta)} \\
&\qquad \leq 4 \rho(\mf{T}) \sup_{z \in \mf{D}_{1/2}} \abs{k_n(\eta_{\tilde{\mu}_{n}}(z) - z) - A^{\beta, \sigma}(z)}
+ 8 \rho(\mf{T}) \sup_{z \in \mf{D}_{1/2}} \abs{(\eta_{\tilde{\mu}_{n}}(z) - z) A^{\beta, \sigma}(z)}.
\end{split}
\]
On $\mf{D}_{1/2}$, $A^{\beta, \sigma}$ is bounded, and $\left(k_n(\eta_{\tilde{\mu}_{n}}(z) - z) - A^{\beta, \sigma}(z)\right)$ and $(\eta_{\tilde{\mu}_{n}}(z) - z)$ converge to zero uniformly. It follows that $\norm{k_n(h \circ \eta_{\tilde{\mu}_{n}} - h) - A^{\beta, \sigma} h'}_\infty \rightarrow 0$.

Denote by $V_n, B$ the operators on $\mc{A}$ given by
\[
V_n h = h \circ \eta_{\tilde{\mu}_{n}},
\]
and
\[
B h = A^{\beta, \sigma} h'.
\]
Then we have just shown that
\[
k_n (V_n - I) h \rightarrow B h
\]
for $h \in \mc{D}$, so in particular $\mc{D}$ is in the domain of $B$. The rest of the argument proceeds as in Proposition~\ref{Prop:Variance}, and implies (by taking $h(z) = z = \eta_{\delta_1}$) that
\[
\norm{\eta_{\tilde{\mu}_{n}^{\circlearrowright k_n}} - \eta_{\nu_{\circlearrowright}^{\gamma, \sigma}}}_\infty \rightarrow 0.
\]
We conclude that
\[
\tilde{\mu}_{n}^{\circlearrowright k_n} \rightarrow \nu_{\circlearrowright}^{\beta, \sigma}
\]
weakly.
\end{proof}

\begin{theorem}
Fix a sequence of probability measures $\set{\mu_{n}}_{n \in \mf{N}}$ on $\mf{T}$, and a sequence of positive integers
$k_{1} < k_{2} < \cdots $. Assume that
\begin{equation}
\label{Eq:beta}
k_n \int_{\mf{T}} \Im (\zeta) \,d\mu_n(\zeta) \rightarrow \beta.
\end{equation}
Then for $\gamma = e^{i \beta}$, the following assertions are equivalent:
\begin{enumerate}
 \item
The sequence $\underbrace{\mu_{n} \utimes \mu_{n} \utimes \cdots \utimes \mu_{n}}_{k_{n}}$ converges weakly to $\nu_{\sutimes}^{\gamma,\sigma}$;
 \item
The sequence $\underbrace{\mu_{n} \circlearrowright \mu_{n} \circlearrowright \cdots \circlearrowright \mu_{n}}_{k_{n}}$ converges weakly to $\nu_{\circlearrowright}^{\beta,\sigma}$;
\end{enumerate}
\end{theorem}

\begin{proof}
Denote $a_n = \int_{\mf{T}} \zeta \,d\mu_n(\zeta)$. Either of the assumptions (a) or (b) implies that $a_n^{k_n} \rightarrow \gamma e^{-\sigma(\mf{T})}$, so that
\[
k_n \log \abs{a_n} \rightarrow - \sigma(\mf{T}).
\]
Since $\abs{a_n} \leq 1$, the series expansion of the logarithm immediately gives $k_n (\abs{a_n} - 1) \rightarrow - \sigma(\mf{T})$.
We also know that $k_n \Im(a_n) \rightarrow \beta$. Combining these we conclude that $k_n (\Re(a_n) - 1) \rightarrow - \sigma(\mf{T})$ and, after a little more work, $k_n(\log(a_n) - (a_n - 1)) \rightarrow 0$, where we again use the principal branch of the logarithm. Note also that the assumption \eqref{Eq:beta} implies that $\mu_n \rightarrow \delta_1$ weakly.

Suppose that $\mu_n^{\sutimes k_n} \rightarrow \nu_{\sutimes}^{\gamma,\sigma}$ in the weak topology. Plugging $z=0$ in equation~\eqref{Eq:Rotation}, using the computation above, and taking the imaginary part, we see that $\ell_n$ in Proposition~\ref{Multiplicative_Proposition} is zero, and it follows from that proposition that $\mu_n^{\circlearrowright k_n} \rightarrow \nu_{\circlearrowright}^{\beta,\sigma}$

Now suppose that $\mu_{n}^{\circlearrowright k_{n}} \rightarrow \nu_{\circlearrowright}^{\beta,\sigma}$. Since the measures $\mu_{n}^{\sutimes k_{n}}$ are supported on $\mathbb{T}$, they form a tight family.  Fix a subsequence converging to a measure $\nu_{\sutimes}^{\gamma',\sigma'}$ (the fact that the point is infinitely divisible may be found in \cite{Wang-Boolean}). Moreover, by comparing the means as above, it follows that $e^{i \beta} = \gamma' = \gamma$. Then by the reverse implication, it follows that $\mu_{n}^{\circlearrowright k_{n}} \rightarrow \nu_{\circlearrowright}^{\beta,\sigma'}$.  We may therefore conclude that $\sigma = \sigma'$, completing the proof.
\end{proof}


\def\cprime{$'$}
\providecommand{\bysame}{\leavevmode\hbox to3em{\hrulefill}\thinspace}
\providecommand{\MR}{\relax\ifhmode\unskip\space\fi MR }
\providecommand{\MRhref}[2]{%
  \href{http://www.ams.org/mathscinet-getitem?mr=#1}{#2}
}
\providecommand{\href}[2]{#2}

\end{document}